\documentclass{amsart}

\usepackage{eurosym}
\usepackage{amsfonts}
\usepackage{graphicx}

\usepackage{paralist}

\usepackage{hyperref}

\newtheorem{theorem}{Theorem}
\theoremstyle{plain}

\newtheorem{definition}{Definition}
\newtheorem{example}{Example}

\newtheorem{proposition}{Proposition}
\newtheorem{remark}{Remark}
\newtheorem{solution}{Solution}

\numberwithin{equation}{section}

\begin{document}
\title[The general linear group of degree $n$ for 3D matrices]{The general
linear group of degree $n$ for 3D matrices $GL(n,n,p;F)$}
\author{Orgest ZAKA}
\address[O. ZAKA]{Department of Mathematics, Faculty of Technical
Science, University of Vlora "Ismail QEMAL", Vlora, Albania}
\email[Orgest ZAKA]{gertizaka@yahoo.com}
\date{December 10, 2018}

\begin{abstract}
In this article we give the meaning of the determinant for 3D matrices with
elements from a field F, and the meaning of 3D inverse matrix. Based on my
previous work titled '3D Matrix Rings', we want to constructed the 'general
linear group of degree $n$ for 3D matrices, which i mark with $GL(n,n,p;F)$'
for 3D-matrices, analog to 'general linear group of degree $n$' known.
\end{abstract}

\maketitle

\section{Introduction, 3D matrix}

This paper comes as a continuation of the ideas that arise based on my
previous work, of the 3D matrix ring with element from an whatever field $%
\mathbf{F}$ (see \cite{[OZ3DM]}). At this point we are making a brief
summary associated with 3D matrices and the 3D matrix presented in \cite%
{[OZ3DM]}. Our objective is to constructed the 'general linear group of
degree $n$ for 3D matrices, which i mark with $GL(n,n,p;F)$' for
3D-matrices, analog to 'general linear group of degree $n$' known. For this
we need new notions, which we will give in the following points.

\begin{definition}
\cite{[OZ3DM]} 3-dimensional $m\times n\times p$ matrix will call, a matrix
which has: m-\textbf{horizontal layers} (analogous to m-rows), n-\textbf{%
vertical page} (analogue with n-columns in the usual matrices) and p-\textbf{%
vertical layers} (p-1 of which are hidden).

The set of these matrixes the write how:%
\begin{equation*}
\mathcal{M}_{m\times n\times p}(\mathbf{F})=\left\{ a_{i,j,k}|a_{i,j,k}\in F-%
\text{field }\forall i=\overline{1,m};\text{ }j=\overline{1,n};\text{ }k=%
\overline{1,p}\right\}
\end{equation*}
\end{definition}

\subsection{ADDITION OF 3D MATRIX}

\begin{definition}
\cite{[OZ3DM]} The addition of two matrices $\mathbf{A,B}\in \mathcal{M}%
_{m\times n\times p}(\mathbf{F})$ we will call the matrix:%
\begin{equation*}
\mathbf{C}_{m\times n\times p}(\mathbf{F})=\left\{
c_{i,j,k}|c_{i,j,k}=a_{i,j,k}+b_{i,j,k},\text{ }\forall i=\overline{1,m};%
\text{ }j=\overline{1,n};\text{ }k=\overline{1,p}\right\}
\end{equation*}

The appearance of the addition of $m\times n\times p,$ 3D matrices will be
as in Figure 1, where matrices $\mathbf{A}$ and $\mathbf{B}$ have the
following appearance,%
\begin{eqnarray*}
\mathbf{A} &=&\left\{ \left( a_{i,j,k}\right) |\text{ }\forall i=\overline{%
1,m};\text{ }j=\overline{1,n};\text{ }k=\overline{1,p}\right\} ; \\
\mathbf{B} &=&\left\{ \left( b_{i,j,k}\right) |\text{ }\forall i=\overline{%
1,m};\text{ }j=\overline{1,n};\text{ }k=\overline{1,p}\right\}
\end{eqnarray*}%

\begin{figure}
\centering
\includegraphics[width=1.0\textwidth]{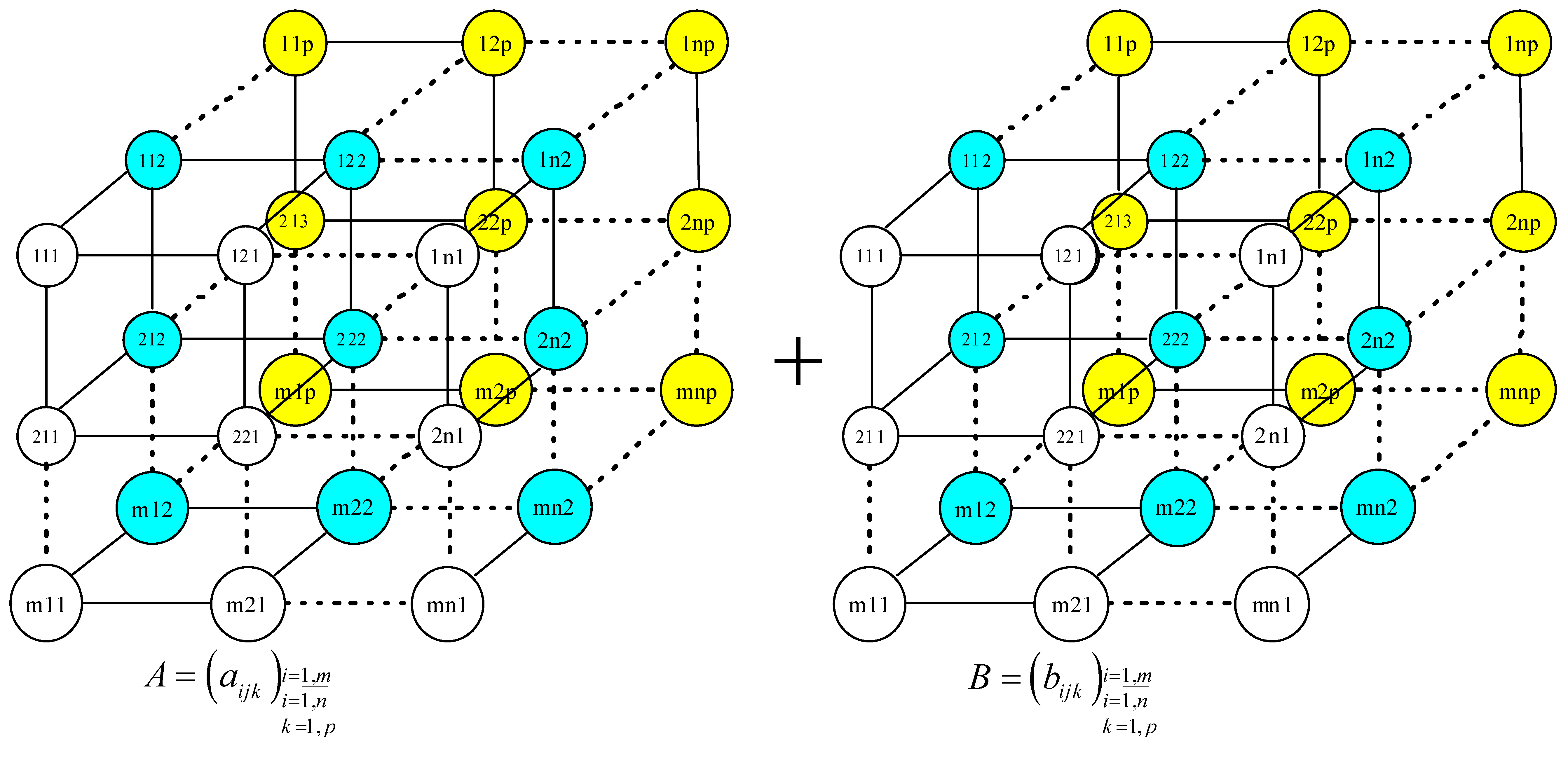}
\label{fig:Fig1}
\caption{The appearance of the addition of 3D matrix}
\end{figure}

\end{definition}

\begin{definition}
\cite{[OZ3DM]}The 3-D, Zero matrix $m\times n\times p,$ we will called the
matrix that has all its elements zero.%
\begin{equation*}
\mathbf{O=O}_{m\times n\times p}=\left\{ \left( 0_{\mathbf{F}}\right)
_{i,j,k}|\text{ }\forall i=\overline{1,m};\text{ }j=\overline{1,n};\text{ }k=%
\overline{1,p}\right\} \in \mathcal{M}_{m\times n\times p}(\mathbf{F})
\end{equation*}
\end{definition}

\begin{definition}
\cite{[OZ3DM]} The opposite matric of an matrice 
\begin{equation*}
\mathbf{A}_{m\times n\times p}=\left\{ \left( a_{i,j,k}\right) |\text{ }%
\forall i=\overline{1,m};\text{ }j=\overline{1,n};\text{ }k=\overline{1,p}%
\right\} \in \mathcal{M}_{m\times n\times p}(\mathbf{F})
\end{equation*}

will, called matrix 
\begin{equation*}
-\mathbf{A}_{m\times n\times p}=\left\{ \left( -a_{i,j,k}\right) |\text{ }%
\forall i=\overline{1,m};\text{ }j=\overline{1,n};\text{ }k=\overline{1,p}%
\right\} \in \mathcal{M}_{m\times n\times p}(\mathbf{F}).
\end{equation*}
\end{definition}

Where $\left( -a_{i,j,k}\right) $ is a opposite element of element $%
a_{i,j,k}\in \mathbf{F,}$ so 
\begin{equation*}
a_{i,j,k}+\left( -a_{i,j,k}\right) =0_{\mathbf{F}},\forall i=\overline{1,m}%
;j=\overline{1,n};k=\overline{1,p},
\end{equation*}

and $\left( \mathbf{F,+,\cdot }\right) $ is field, which satisfies the
condition

\begin{eqnarray*}
\mathbf{A}_{m\times n\times p}+\left( -\mathbf{A}_{m\times n\times p}\right)
&=&\left\{ \left( a_{i,j,k}\right) +\left( -a_{i,j,k}\right) |\text{ }%
\forall i=\overline{1,m};\text{ }j=\overline{1,n};\text{ }k=\overline{1,p}%
\right\} \\
&=&\left\{ \left( 0_{\mathbf{F}}\right) _{i,j,k}|\text{ }\forall i=\overline{%
1,m};\text{ }j=\overline{1,n};\text{ }k=\overline{1,p}\right\} \\
&=&\mathbf{O}_{m\times n\times p}=\mathbf{O}
\end{eqnarray*}

\begin{theorem}
\cite{[OZ3DM]}$\left( \mathcal{M}_{m\times n\times p}(\mathbf{F}),+\right) $
is abelian grup.
\end{theorem}

\subsection{THE MULTIPLICATION OF $n\times n\times p$ MATRICES}

In the same way, as have the meaning of a 3D $3\times 3\times 3$ matrix
multiplication to \cite{[OZ3DM]}, we give the definition of 3D matrix
multiplication for \ $\mathbf{A,B}\in \mathcal{M}_{n\times n\times p}(%
\mathbf{F})$.

\begin{definition}
\cite{[OZ3DM]} The multiplication of two matrices $\mathbf{A,B}\in \mathcal{M%
}_{n\times n\times p}(\mathbf{F})$ we will call the matrix $\mathbf{C=A\odot
B}\in \mathcal{M}_{n\times n\times p}(\mathbf{F}),$ calculated as follows:

Matrices will normally have the appearance:%
\begin{equation*}
\forall \mathbf{A=}\left[ 
\begin{array}{c}
\left( 
\begin{array}{cccc}
a_{11p} & a_{12p} & \cdots & a_{1np} \\ 
a_{21p} & a_{22p} & \cdots & a_{2np} \\ 
\vdots & \vdots & \ddots & \vdots \\ 
a_{n1p} & a_{n2p} & \cdots & a_{nnp}%
\end{array}%
\right) \\ 
\mathbf{\vdots } \\ 
\left( 
\begin{array}{cccc}
a_{112} & a_{122} & \cdots & a_{1n2} \\ 
a_{212} & a_{222} & \cdots & a_{2n2} \\ 
\vdots & \vdots & \ddots & \vdots \\ 
a_{n12} & a_{n22} & \cdots & a_{nn2}%
\end{array}%
\right) \\ 
\left( 
\begin{array}{cccc}
a_{111} & a_{121} & \cdots & a_{1n1} \\ 
a_{211} & a_{221} & \cdots & a_{2n1} \\ 
\vdots & \vdots & \ddots & \vdots \\ 
a_{n11} & a_{n21} & \cdots & a_{nn1}%
\end{array}%
\right)%
\end{array}%
\right] ;\mathbf{B=}\left[ 
\begin{array}{c}
\left( 
\begin{array}{cccc}
b_{11p} & b_{12p} & \cdots & b_{1np} \\ 
b_{21p} & b_{22p} & \cdots & b_{2np} \\ 
\vdots & \vdots & \ddots & \vdots \\ 
b_{n1p} & b_{n2p} & \cdots & b_{nnp}%
\end{array}%
\right) \\ 
\mathbf{\vdots } \\ 
\left( 
\begin{array}{cccc}
b_{112} & b_{122} & \cdots & b_{1n2} \\ 
b_{212} & b_{222} & \cdots & b_{2n2} \\ 
\vdots & \vdots & \ddots & \vdots \\ 
b_{n12} & b_{n22} & \cdots & b_{nn2}%
\end{array}%
\right) \\ 
\left( 
\begin{array}{cccc}
b_{111} & b_{121} & \cdots & b_{1n1} \\ 
b_{211} & b_{221} & \cdots & b_{2n1} \\ 
\vdots & \vdots & \ddots & \vdots \\ 
b_{n11} & b_{n21} & \cdots & b_{nn1}%
\end{array}%
\right)%
\end{array}%
\right]
\end{equation*}

If we write more briefly%
\begin{equation*}
\mathbf{A=}\left[ 
\begin{array}{c}
\mathbf{A}_{p} \\ 
\vdots \\ 
\mathbf{A}_{2} \\ 
\mathbf{A}_{1}%
\end{array}%
\right] ,\mathbf{B=}\left[ 
\begin{array}{c}
\mathbf{B}_{p} \\ 
\vdots \\ 
\mathbf{B}_{2} \\ 
\mathbf{B}_{1}%
\end{array}%
\right] \in \mathcal{M}_{n\times n\times p}(\mathbf{F})
\end{equation*}

where $\mathbf{A}_{i}$ and $\mathbf{B}_{i}$ are the $n\times n$ matrices $%
\forall i=\overline{1,p}.$ Then 
\begin{equation*}
\mathbf{C=A\odot B=\left[ 
\begin{array}{c}
\mathbf{A}_{p} \\ 
\vdots \\ 
\mathbf{A}_{2} \\ 
\mathbf{A}_{1}%
\end{array}%
\right] \odot }\left[ 
\begin{array}{c}
\mathbf{B}_{p} \\ 
\vdots \\ 
\mathbf{B}_{2} \\ 
\mathbf{B}_{1}%
\end{array}%
\right] =\mathbf{\left[ 
\begin{array}{c}
\mathbf{A}_{p}\ast \mathbf{B}_{p} \\ 
\vdots \\ 
\mathbf{A}_{2}\ast \mathbf{B}_{2} \\ 
\mathbf{A}_{1}\ast \mathbf{B}_{1}%
\end{array}%
\right] }
\end{equation*}

where action $"\ast "$ is the usual multiplication of matrices.
\end{definition}

\section{Multi-Scalars and Multiplication of 3D matrices with multi-scalar}

In this section we will introduce the concept of '\textit{multi-scalar}',
and we will give a clear idea of the multiplication of 3D matrices with
multi-scalar.

\begin{definition}
Multi-scalar will call one $1\times 1\times p,$ 3D matrix.
\end{definition}

\begin{remark}
A multi-scalar $\ a_{1\times 1\times p}=\left[ 
\begin{array}{c}
\left( \alpha _{11p}\right) \\ 
\vdots \\ 
\left( \alpha _{112}\right) \\ 
\left( \alpha _{111}\right)%
\end{array}%
\right] ,$ we will call "\textbf{absolutely}" different from zero only if $%
\alpha _{11i}\neq 0_{\mathbf{F}},\forall i=\overline{1,p}.$ For the "\textbf{%
absolutely zero}" multi-scalar we will use the note $\widetilde{^{a}0_{%
\mathbf{F}}},$ wich is $\ \widetilde{^{a}0_{\mathbf{F}}}=\left[ 
\begin{array}{c}
\left( 0_{\mathbf{F},p}\right) \\ 
\vdots \\ 
\left( 0_{\mathbf{F},2}\right) \\ 
\left( 0_{\mathbf{F,}1}\right)%
\end{array}%
\right] .$
\end{remark}

Let's have a \textbf{multi-scalar \ }

$\ \ $%
\begin{equation*}
a_{1\times 1\times p}=\left[ 
\begin{array}{c}
\left( \alpha _{11p}\right) \\ 
\vdots \\ 
\left( \alpha _{112}\right) \\ 
\left( \alpha _{111}\right)%
\end{array}%
\right] ,
\end{equation*}

and 3D matrix

$\ $%
\begin{equation*}
\mathbf{A}_{m\times n\times p}\mathbf{=}\left[ 
\begin{array}{c}
\left( 
\begin{array}{cccc}
a_{11p} & a_{12p} & \cdots & a_{1np} \\ 
a_{21p} & a_{22p} & \cdots & a_{2np} \\ 
\vdots & \vdots & \ddots & \vdots \\ 
a_{m1p} & a_{m2p} & \cdots & a_{mnp}%
\end{array}%
\right) \\ 
\vdots \\ 
\left( 
\begin{array}{cccc}
a_{112} & a_{122} & \cdots & a_{1n2} \\ 
a_{212} & a_{222} & \cdots & a_{2n2} \\ 
\vdots & \vdots & \ddots & \vdots \\ 
a_{m12} & a_{m22} & \cdots & a_{mn2}%
\end{array}%
\right) \\ 
\left( 
\begin{array}{cccc}
a_{111} & a_{121} & \cdots & a_{1n1} \\ 
a_{211} & a_{221} & \cdots & a_{2n1} \\ 
\vdots & \vdots & \ddots & \vdots \\ 
a_{m11} & a_{m21} & \cdots & a_{mn1}%
\end{array}%
\right)%
\end{array}%
\right] ,
\end{equation*}

\begin{definition}
\ Multiplication of 3D-matrix $\mathbf{A}_{m\times n\times p}\in \mathcal{M}%
_{m\times n\times p}(\mathbf{F})$ \ with multi-scalar $a_{1\times 1\times p}$
will we call the 3D matrix $\mathbf{B}_{m\times n\times p}\in \mathcal{M}%
_{m\times n\times p}(\mathbf{F}),$ calculated as follows:%
\begin{eqnarray*}
\mathbf{B}_{m\times n\times p} &\mathbf{=}&a_{1\times 1\times p}\ast \mathbf{%
A}_{m\times n\times p}= \\
&=&\left[ 
\begin{array}{c}
\alpha _{11p} \\ 
\vdots \\ 
\alpha _{112} \\ 
\alpha _{111}%
\end{array}%
\right] \ast \left[ 
\begin{array}{c}
\left( 
\begin{array}{cccc}
a_{11p} & a_{12p} & \cdots & a_{1np} \\ 
a_{21p} & a_{22p} & \cdots & a_{2np} \\ 
\vdots & \vdots & \ddots & \vdots \\ 
a_{m1p} & a_{m2p} & \cdots & a_{mnp}%
\end{array}%
\right) \\ 
\vdots \\ 
\left( 
\begin{array}{cccc}
a_{112} & a_{122} & \cdots & a_{1n2} \\ 
a_{212} & a_{222} & \cdots & a_{2n2} \\ 
\vdots & \vdots & \ddots & \vdots \\ 
a_{m12} & a_{m22} & \cdots & a_{mn2}%
\end{array}%
\right) \\ 
\left( 
\begin{array}{cccc}
a_{111} & a_{121} & \cdots & a_{1n1} \\ 
a_{211} & a_{221} & \cdots & a_{2n1} \\ 
\vdots & \vdots & \ddots & \vdots \\ 
a_{m11} & a_{m21} & \cdots & a_{mn1}%
\end{array}%
\right)%
\end{array}%
\right]
\end{eqnarray*}%
\begin{equation*}
=\left[ 
\begin{array}{c}
\left( 
\begin{array}{cccc}
\alpha _{11p}\cdot a_{11p} & \alpha _{11p}\cdot a_{12p} & \cdots & \alpha
_{11p}\cdot a_{1np} \\ 
\alpha _{11p}\cdot a_{21p} & \alpha _{11p}\cdot a_{22p} & \cdots & \alpha
_{11p}\cdot a_{2np} \\ 
\vdots & \vdots & \ddots & \vdots \\ 
\alpha _{11p}\cdot a_{m1p} & \alpha _{11p}\cdot a_{m2p} & \cdots & \alpha
_{11p}\cdot a_{mnp}%
\end{array}%
\right) \\ 
\mathbf{\vdots } \\ 
\left( 
\begin{array}{cccc}
\alpha _{112}\cdot a_{112} & \alpha _{112}\cdot a_{122} & \cdots & \alpha
_{112}\cdot a_{1n2} \\ 
\alpha _{112}\cdot a_{212} & \alpha _{112}\cdot a_{222} & \cdots & \alpha
_{112}\cdot a_{2n2} \\ 
\vdots & \vdots & \ddots & \vdots \\ 
\alpha _{112}\cdot a_{m12} & \alpha _{112}\cdot a_{m22} & \cdots & \alpha
_{112}\cdot a_{mn2}%
\end{array}%
\right) \\ 
\left( 
\begin{array}{cccc}
\alpha _{111}\cdot a_{111} & \alpha _{111}\cdot a_{121} & \cdots & \alpha
_{111}\cdot a_{1n1} \\ 
\alpha _{111}\cdot a_{211} & \alpha _{111}\cdot a_{221} & \cdots & \alpha
_{111}\cdot a_{2n1} \\ 
\vdots & \vdots & \ddots & \vdots \\ 
\alpha _{111}\cdot a_{m11} & \alpha _{111}\cdot a_{m21} & \cdots & \alpha
_{111}\cdot a_{mn1}%
\end{array}%
\right)%
\end{array}%
\right]
\end{equation*}
\end{definition}

So multiplication of the 3D matrix with a multi-scalar is a function

\begin{equation*}
\ast :\mathcal{M}_{1\times 1\times p}(\mathbf{F})\times \mathcal{M}_{m\times
n\times p}(\mathbf{F})\longrightarrow \mathcal{M}_{m\times n\times p}(%
\mathbf{F}).
\end{equation*}

\begin{example}
Let's have a multi-scalar \ 
\begin{equation*}
a_{1\times 1\times 3}=\left[ 
\begin{array}{c}
\left( 2\right) \\ 
\left( 5\right) \\ 
\left( 3\right)%
\end{array}%
\right] \ \ 
\end{equation*}%
$\ $

and 3D matrix 
\begin{equation*}
\mathbf{A}_{2\times 2\times 3}=\left[ 
\begin{array}{c}
\left( 
\begin{array}{cc}
2 & 3 \\ 
4 & 5%
\end{array}%
\right) \\ 
\left( 
\begin{array}{cc}
5 & 0 \\ 
9 & 1%
\end{array}%
\right) \\ 
\left( 
\begin{array}{cc}
1 & 4 \\ 
5 & 3%
\end{array}%
\right)%
\end{array}%
\right] .
\end{equation*}

Matrix obtained by multiplying the 3D matrix $A_{2\times 2\times 3}$ with
multi-scalar $a_{1\times 1\times 3},$ it is the matrix:%
\begin{eqnarray*}
\mathbf{B}_{2\times 2\times 3} &=&a_{1\times 1\times 3}\ast \mathbf{A}%
_{2\times 2\times 3} \\
&=&\left[ 
\begin{array}{c}
\left( 2\right) \\ 
\left( 5\right) \\ 
\left( 3\right)%
\end{array}%
\right] \ast \left[ 
\begin{array}{c}
\left( 
\begin{array}{cc}
2 & 3 \\ 
4 & 5%
\end{array}%
\right) \\ 
\left( 
\begin{array}{cc}
5 & 0 \\ 
9 & 1%
\end{array}%
\right) \\ 
\left( 
\begin{array}{cc}
1 & 4 \\ 
5 & 3%
\end{array}%
\right)%
\end{array}%
\right] = \\
&=&\left[ 
\begin{array}{c}
2\cdot \left( 
\begin{array}{cc}
2 & 3 \\ 
4 & 5%
\end{array}%
\right) \\ 
5\cdot \left( 
\begin{array}{cc}
5 & 0 \\ 
9 & 1%
\end{array}%
\right) \\ 
3\cdot \left( 
\begin{array}{cc}
1 & 4 \\ 
5 & 3%
\end{array}%
\right)%
\end{array}%
\right] =\left[ 
\begin{array}{c}
\left( 
\begin{array}{cc}
4 & 6 \\ 
8 & 10%
\end{array}%
\right) \\ 
\left( 
\begin{array}{cc}
25 & 0 \\ 
45 & 5%
\end{array}%
\right) \\ 
\left( 
\begin{array}{cc}
3 & 12 \\ 
15 & 9%
\end{array}%
\right)%
\end{array}%
\right] .
\end{eqnarray*}
\end{example}

\subsection{DETERMINANTS OF 3D-MATRICES}

Regarding the determinant we will only talk about form matrices $\mathbf{A}%
_{n\times n\times p}\in \mathcal{M}_{n\times n\times p}(\mathbf{F}),$ ie for
matrices that \emph{vertical layers} have square matrices.

\begin{definition}
Determinant of the matrix $\mathbf{A}_{n\times n\times p}\in \mathcal{M}%
_{n\times n\times p}(\mathbf{F}),$ we will call the \textbf{multi-scalar}%
\begin{equation*}
\det \left( \mathbf{A}_{n\times n\times p}\right) =\left[ 
\begin{array}{c}
\det \left( \mathbf{A}_{n\times n,p}\right) \\ 
\mathbf{\vdots } \\ 
\det \left( \mathbf{A}_{n\times n,2}\right) \\ 
\det \left( \mathbf{A}_{n\times n,1}\right)%
\end{array}%
\right]
\end{equation*}

where%
\begin{equation*}
\det \left( \mathbf{A}_{n\times n,1}\right) =\det \left( 
\begin{array}{cccc}
a_{111} & a_{121} & \cdots & a_{1n1} \\ 
a_{211} & a_{221} & \cdots & a_{2n1} \\ 
\vdots & \vdots & \ddots & \vdots \\ 
a_{n11} & a_{n21} & \cdots & a_{nn1}%
\end{array}%
\right) ;
\end{equation*}%
\begin{equation*}
\det \left( \mathbf{A}_{n\times n,2}\right) =\det \left( 
\begin{array}{cccc}
a_{112} & a_{122} & \cdots & a_{1n2} \\ 
a_{212} & a_{222} & \cdots & a_{2n2} \\ 
\vdots & \vdots & \ddots & \vdots \\ 
a_{n12} & a_{n22} & \cdots & a_{nn2}%
\end{array}%
\right) ;...;
\end{equation*}%
\begin{equation*}
\det \left( \mathbf{A}_{n\times n,p}\right) =\det \left( 
\begin{array}{cccc}
a_{11p} & a_{12p} & \cdots & a_{1np} \\ 
a_{21p} & a_{22p} & \cdots & a_{2np} \\ 
\vdots & \vdots & \ddots & \vdots \\ 
a_{n1p} & a_{n2p} & \cdots & a_{nnp}%
\end{array}%
\right)
\end{equation*}
\end{definition}

Referring to the multi-scalar note '\emph{absolutely different from zero}',
we say that for a 3D matrix $\mathbf{A}_{n\times n\times p},$ \ 

\begin{equation*}
\det \left( \mathbf{A}_{n\times n\times p}\right) \neq \widetilde{^{a}0_{%
\mathbf{F}}}\Longleftrightarrow \det \left( \mathbf{A}_{n\times n,i}\right)
\neq 0_{\mathbf{F}},\forall i=\overline{1,p}.
\end{equation*}

\begin{example}
Determinant of the matrix $\mathbf{A}_{2\times 2\times 3},$ of example 1, is
multi-scalar:
\end{example}

\begin{equation*}
\det \left( \mathbf{A}_{2\times 2\times 3}\right) =\left[ 
\begin{array}{c}
\det \left( 
\begin{array}{cc}
2 & 3 \\ 
4 & 5%
\end{array}%
\right) \\ 
\det \left( 
\begin{array}{cc}
5 & 0 \\ 
9 & 1%
\end{array}%
\right) \\ 
\det \left( 
\begin{array}{cc}
1 & 4 \\ 
5 & 3%
\end{array}%
\right)%
\end{array}%
\right] =\left[ 
\begin{array}{c}
\left( -2\right) \\ 
\left( 5\right) \\ 
\left( -17\right)%
\end{array}%
\right]
\end{equation*}

\begin{definition}
\emph{Inverted} of the multi-scalar $a_{1\times 1\times p}=\left[ 
\begin{array}{c}
\left( \alpha _{11p}\right) \\ 
\vdots \\ 
\left( \alpha _{112}\right) \\ 
\left( \alpha _{111}\right)%
\end{array}%
\right] ,$ will called the multi-scalar:%
\begin{equation*}
\widehat{a_{1\times 1\times p}}=\left[ 
\begin{array}{c}
\left( \frac{1}{\alpha _{11p}}\right) \\ 
\vdots \\ 
\left( \frac{1}{\alpha _{112}}\right) \\ 
\left( \frac{1}{\alpha _{111}}\right)%
\end{array}%
\right] .
\end{equation*}
\end{definition}

\begin{example}
\emph{Inverted} of the multi-scalar \ $a_{1\times 1\times 3}=\left[ 
\begin{array}{c}
\left( 2\right) \\ 
\left( 5\right) \\ 
\left( 3\right)%
\end{array}%
\right] ,$ is the multi-scalar:%
\begin{equation*}
\widehat{a_{1\times 1\times 3}}=\left[ 
\begin{array}{c}
\left( \frac{1}{2}\right) \\ 
\left( \frac{1}{5}\right) \\ 
\left( \frac{1}{3}\right)%
\end{array}%
\right] .
\end{equation*}
\end{example}

\section{THE MULTIPLICATION GROUP OF 3D MATRICES}

Referring to \cite{[OZ3DM]}, we have a 3D Matrix Ring, and in that paper we
have shown the possibility of unitary ring. The rest of the assertions that
lead us from unitary ring to a Skew-Field are summing up in this

\begin{theorem}
The structure $\left( \mathcal{M}_{n\times n\times p}(\mathbf{F}),\odot
\right) $ is a unitary semi-group.
\end{theorem}

\begin{proof}
We show first that:%
\begin{equation*}
\forall \mathbf{A},\mathbf{B},\mathbf{C}\in \mathcal{M}_{n\times n\times p}(%
\mathbf{F}),\left[ \mathbf{A}\odot \mathbf{B}\right] \odot \mathbf{C=A}\odot %
\left[ \mathbf{B}\odot \mathbf{C}\right]
\end{equation*}

truly,

$\left[ \mathbf{A}\odot \mathbf{B}\right] \odot \mathbf{C=}\left( \left[ 
\begin{array}{c}
\mathbf{A}_{p} \\ 
\vdots \\ 
\mathbf{A}_{2} \\ 
\mathbf{A}_{1}%
\end{array}%
\right] \odot \left[ 
\begin{array}{c}
\mathbf{B}_{p} \\ 
\vdots \\ 
\mathbf{B}_{2} \\ 
\mathbf{B}_{1}%
\end{array}%
\right] \right) \odot \left[ 
\begin{array}{c}
\mathbf{C}_{p} \\ 
\vdots \\ 
\mathbf{C}_{2} \\ 
\mathbf{C}_{1}%
\end{array}%
\right] =\left[ 
\begin{array}{c}
\mathbf{A}_{p}\cdot \mathbf{B}_{p} \\ 
\vdots \\ 
\mathbf{A}_{2}\cdot \mathbf{B}_{2} \\ 
\mathbf{A}_{1}\cdot \mathbf{B}_{1}%
\end{array}%
\right] \odot \left[ 
\begin{array}{c}
\mathbf{C}_{p} \\ 
\vdots \\ 
\mathbf{C}_{2} \\ 
\mathbf{C}_{1}%
\end{array}%
\right] =\left[ 
\begin{array}{c}
\left( \mathbf{A}_{p}\cdot \mathbf{B}_{p}\right) \cdot \mathbf{C}_{p} \\ 
\vdots \\ 
\left( \mathbf{A}_{2}\cdot \mathbf{B}_{2}\right) \cdot \mathbf{C}_{2} \\ 
\left( \mathbf{A}_{1}\cdot \mathbf{B}_{1}\right) \cdot \mathbf{C}_{1}%
\end{array}%
\right] =\left[ 
\begin{array}{c}
\mathbf{A}_{p}\cdot \left( \mathbf{B}_{p}\cdot \mathbf{C}_{p}\right) \\ 
\vdots \\ 
\mathbf{A}_{2}\cdot \left( \mathbf{B}_{2}\cdot \mathbf{C}_{2}\right) \\ 
\mathbf{A}_{1}\cdot \left( \mathbf{B}_{1}\cdot \mathbf{C}_{1}\right)%
\end{array}%
\right] =\left[ 
\begin{array}{c}
\mathbf{A}_{p} \\ 
\vdots \\ 
\mathbf{A}_{2} \\ 
\mathbf{A}_{1}%
\end{array}%
\right] \odot \left[ 
\begin{array}{c}
\mathbf{B}_{p}\cdot \mathbf{C}_{p} \\ 
\vdots \\ 
\mathbf{B}_{2}\cdot \mathbf{C}_{2} \\ 
\mathbf{B}_{1}\cdot \mathbf{C}_{1}%
\end{array}%
\right] $

$=\left[ 
\begin{array}{c}
\mathbf{A}_{p} \\ 
\vdots \\ 
\mathbf{A}_{2} \\ 
\mathbf{A}_{1}%
\end{array}%
\right] \odot \left( \left[ 
\begin{array}{c}
\mathbf{B}_{p} \\ 
\vdots \\ 
\mathbf{B}_{2} \\ 
\mathbf{B}_{1}%
\end{array}%
\right] \odot \left[ 
\begin{array}{c}
\mathbf{C}_{p} \\ 
\vdots \\ 
\mathbf{C}_{2} \\ 
\mathbf{C}_{1}%
\end{array}%
\right] \right) =\mathbf{A}\odot \left[ \mathbf{B}\odot \mathbf{C}\right] .$

From the definition of multiplication to 3D matrices it is clear that:%
\begin{equation*}
\ \bigskip \mathbf{I}_{n\times n\times p}=\left[ 
\begin{array}{c}
\mathbf{I}_{n\times n,p} \\ 
\vdots \\ 
\mathbf{I}_{n\times n,2} \\ 
\mathbf{I}_{n\times n,1}%
\end{array}%
\right] ,
\end{equation*}

is unit element of $\mathcal{M}_{n\times n\times p}(\mathbf{F})$, related to
multiplication.
\end{proof}

\begin{remark}
We mark with $\mathcal{M}_{n\times n\times p}^{\ast }(\mathbf{F}),$ 3D
matrix with determinants '\textbf{absolutely different from zero}', ie $\det
\left( \mathbf{A}_{n\times n\times p}\right) \neq \widetilde{^{a}0_{\mathbf{F%
}}}$, (So \textbf{vertical layers} of 3D matrix, are 2D non-singular
matrices.).
\end{remark}

\begin{proposition}
The set of \ $\mathcal{M}_{n\times n\times p}^{\ast }(\mathbf{F}),$ is
closed regarding multiplication.

Well! 
\begin{eqnarray*}
\odot &:&\mathcal{M}_{n\times n\times p}^{\ast }(\mathbf{F})\times \mathcal{M%
}_{n\times n\times p}^{\ast }(\mathbf{F})\rightarrow \mathcal{M}_{n\times
n\times p}^{\ast }(\mathbf{F}),\  \\
\forall \mathbf{A,B} &\in &\mathcal{M}_{n\times n\times p}^{\ast }(\mathbf{F}%
)\Rightarrow \mathbf{A\odot B}\in \mathcal{M}_{n\times n\times p}^{\ast }(%
\mathbf{F}).
\end{eqnarray*}

\begin{proof}
Let's have 
\begin{equation*}
\mathbf{A,B}\in \mathcal{M}_{n\times n\times p}^{\ast }(\mathbf{F}%
)\Rightarrow \det \left( \mathbf{A}\right) \neq \widetilde{^{a}0_{\mathbf{F}}%
}\text{ and }\det \left( \mathbf{B}\right) \neq \widetilde{^{a}0_{\mathbf{F}}%
}
\end{equation*}
\begin{equation*}
\Leftrightarrow \det \left( \mathbf{A}_{n\times n,i}\right) \neq 0_{\mathbf{F%
}}\text{ }and\text{ }\det \left( \mathbf{B}_{n\times n,i}\right) \neq 0_{%
\mathbf{F}},,\forall i=\overline{1,p}
\end{equation*}

So all 2D matrices $\mathbf{A}_{n\times n,i}$ and $\mathbf{B}_{n\times n,i}$%
, $\forall i=\overline{1,p}$ , are non-singular matrices. A well-known
result in Linear algebra, but also mentioned in \cite{[GS4Ed]}, in pp.230,
and \cite{[ShALA3ED]} in pp317, which is 
\begin{equation*}
"\forall A,B\in \mathcal{M}_{n\times n}(\mathbf{F}),\ \det \left( A\cdot
B\right) =\det \left( A\right) \cdot \det \left( B\right) "
\end{equation*}

We use this result as the \textbf{vertical layers} of 3D matrix are
2D-matries, and so we have that:

$\det \left( \mathbf{A}\odot \mathbf{B}\right) \mathbf{=}\det \left( \left[ 
\begin{array}{c}
\mathbf{A}_{n\times n,p} \\ 
\vdots \\ 
\mathbf{A}_{n\times n,2} \\ 
\mathbf{A}_{n\times n,1}%
\end{array}%
\right] \odot \left[ 
\begin{array}{c}
\mathbf{B}_{n\times n,p} \\ 
\vdots \\ 
\mathbf{B}_{n\times n,2} \\ 
\mathbf{B}_{n\times n,1}%
\end{array}%
\right] \right) =\det \left[ 
\begin{array}{c}
\mathbf{A}_{n\times n,p}\cdot \mathbf{B}_{n\times n,p} \\ 
\vdots \\ 
\mathbf{A}_{n\times n,2}\cdot \mathbf{B}_{n\times n,2} \\ 
\mathbf{A}_{n\times n,1}\cdot \mathbf{B}_{n\times n,1}%
\end{array}%
\right] =\left[ 
\begin{array}{c}
\det \left( \mathbf{A}_{n\times n,p}\cdot \mathbf{B}_{n\times n,p}\right) \\ 
\vdots \\ 
\det \left( \mathbf{A}_{n\times n,2}\cdot \mathbf{B}_{n\times n,2}\right) \\ 
\det \left( \mathbf{A}_{n\times n,1}\cdot \mathbf{B}_{n\times n,1}\right)%
\end{array}%
\right] \overset{[2]}{=}\left[ 
\begin{array}{c}
\det \left( \mathbf{A}_{n\times n,p}\right) \cdot \det \left( \mathbf{B}%
_{n\times n,p}\right) \\ 
\vdots \\ 
\det \left( \mathbf{A}_{n\times n,2}\right) \cdot \det \left( \mathbf{B}%
_{n\times n,2}\right) \\ 
\det \left( \mathbf{A}_{n\times n,1}\right) \cdot \det \left( \mathbf{B}%
_{n\times n,1}\right)%
\end{array}%
\right] =$

$=\left[ 
\begin{array}{c}
\det \left( \mathbf{A}_{n\times n,p}\right) \\ 
\vdots \\ 
\det \left( \mathbf{A}_{n\times n,2}\right) \\ 
\det \left( \mathbf{A}_{n\times n,1}\right)%
\end{array}%
\right] \odot \left[ 
\begin{array}{c}
\det \left( \mathbf{B}_{n\times n,p}\right) \\ 
\vdots \\ 
\det \left( \mathbf{B}_{n\times n,2}\right) \\ 
\det \left( \mathbf{B}_{n\times n,1}\right)%
\end{array}%
\right] =\det \left( \mathbf{A}\right) \odot \det \left( \mathbf{B}\right)
\neq \widetilde{^{a}0_{\mathbf{F}}}\Rightarrow $

$\Rightarrow \mathbf{A\odot B}\in \mathcal{M}_{n\times n\times p}^{\ast }(%
\mathbf{F}).$
\end{proof}
\end{proposition}

\begin{theorem}
The structure $\left( \mathcal{M}_{n\times n\times p}^{\ast }(\mathbf{F}%
),\odot \right) $ is a Group.

Set of 3D matrix $\mathbf{A}_{n\times n\times p}\in \mathcal{M}_{n\times
n\times p}(\mathbf{F}),$ with determinants '\textbf{absolutely different
from zero}', ie $\det \left( \mathbf{A}_{n\times n\times p}\right) \neq 
\widetilde{^{a}0_{\mathbf{F}}}$, associated with ordinary multiplication is
a Group.

\begin{proof}
\textbf{1.} It is clear that the set $\mathcal{M}_{n\times n\times p}^{\ast
}(\mathbf{F})$ is sub-set of $\mathcal{M}_{n\times n\times p}(\mathbf{F})$,
and in the foregoing assertion we showed that the multiplication is closed
in this set, so we have that $\left( \mathcal{M}_{n\times n\times p}^{\ast }(%
\mathbf{F}),\odot \right) $ is subsemigroup of semigroup $\left( \mathcal{M}%
_{n\times n\times p}(\mathbf{F}),\odot \right) ,$ see \cite{[LSGTM]}, \cite%
{[GPASG]}.

\textbf{2.} It is clear that the $\mathbf{I}_{n\times n\times p}\in \mathcal{%
M}_{n\times n\times p}^{\ast }(\mathbf{F})$, after

\begin{equation*}
\det \left( \mathbf{I}_{n\times n\times p}\right) =\left[ 
\begin{array}{c}
\det \left( \mathbf{I}_{n\times n,p}\right) \\ 
\vdots \\ 
\det \left( \mathbf{I}_{n\times n,2}\right) \\ 
\det \left( \mathbf{I}_{n\times n,1}\right)%
\end{array}%
\right] =\left[ 
\begin{array}{c}
1_{\mathbf{F}} \\ 
\vdots \\ 
1_{\mathbf{F}} \\ 
1_{\mathbf{F}}%
\end{array}%
\right] \neq \widetilde{^{a}0_{\mathbf{F}}}.
\end{equation*}

\textbf{3.} Show that: $\forall \mathbf{A}_{n\times n\times p}\in \mathcal{M}%
_{n\times n\times p}^{\ast }(\mathbf{F}),($such that $\det \left( \mathbf{A}%
_{n\times n\times p}\right) \neq \widetilde{^{a}0_{\mathbf{F}}})$, $\exists 
\mathbf{A}_{n\times n\times p}^{-1}\in \mathcal{M}_{n\times n\times p}^{\ast
}(\mathbf{F}),$ \ such that%
\begin{equation*}
\mathbf{A}_{n\times n\times p}\odot \mathbf{A}_{n\times n\times p}^{-1}=%
\mathbf{I}_{n\times n\times p}
\end{equation*}

Let's have 
\begin{eqnarray*}
\mathbf{A}_{n\times n\times p} &=&\left[ 
\begin{array}{c}
\mathbf{A}_{n\times n,p} \\ 
\vdots \\ 
\mathbf{A}_{n\times n,2} \\ 
\mathbf{A}_{n\times n,1}%
\end{array}%
\right] \in \mathcal{M}_{n\times n\times p}^{\ast }(\mathbf{F})\Rightarrow \\
&\Rightarrow &\det \left( \mathbf{A}_{n\times n\times p}\right) \neq 
\widetilde{^{a}0_{\mathbf{F}}}\Leftrightarrow \det \left( \mathbf{A}%
_{n\times n,i}\right) \neq 0_{\mathbf{F}},\forall i=\overline{1,p}
\end{eqnarray*}

So we have, that: $\forall \mathbf{A}_{n\times n,i}$ (as a 2D matrix), $%
\exists \mathbf{A}_{n\times n,i}^{-1}$ such that $\mathbf{A}_{n\times
n,i}\cdot \mathbf{A}_{n\times n,i}^{-1}=\mathbf{I}_{n\times n,i},\forall i=%
\overline{1,p}.$ From this, we can write that, the inverse of 3D matrix $%
\mathbf{A}_{n\times n\times p}=\left[ 
\begin{array}{c}
\mathbf{A}_{n\times n,p} \\ 
\vdots \\ 
\mathbf{A}_{n\times n,2} \\ 
\mathbf{A}_{n\times n,1}%
\end{array}%
\right] $ is a 3D matrix $\mathbf{A}_{n\times n\times p}^{-1}=\left[ 
\begin{array}{c}
\mathbf{A}_{n\times n,p}^{-1} \\ 
\vdots \\ 
\mathbf{A}_{n\times n,2}^{-1} \\ 
\mathbf{A}_{n\times n,1}^{-1}%
\end{array}%
\right] ,$ because the 
\begin{eqnarray*}
\mathbf{A}_{n\times n\times p}\odot \mathbf{A}_{n\times n\times p}^{-1} &=&%
\left[ 
\begin{array}{c}
\mathbf{A}_{n\times n,p} \\ 
\vdots \\ 
\mathbf{A}_{n\times n,2} \\ 
\mathbf{A}_{n\times n,1}%
\end{array}%
\right] \odot \left[ 
\begin{array}{c}
\mathbf{A}_{n\times n,p}^{-1} \\ 
\vdots \\ 
\mathbf{A}_{n\times n,2}^{-1} \\ 
\mathbf{A}_{n\times n,1}^{-1}%
\end{array}%
\right] = \\
&=&\left[ 
\begin{array}{c}
\mathbf{A}_{n\times n,p}\cdot \mathbf{A}_{n\times n,p}^{-1} \\ 
\vdots \\ 
\mathbf{A}_{n\times n,2}\cdot \mathbf{A}_{n\times n,2}^{-1} \\ 
\mathbf{A}_{n\times n,1}\cdot \mathbf{A}_{n\times n,1}^{-1}%
\end{array}%
\right] \\
&=&\left[ 
\begin{array}{c}
\mathbf{I}_{n\times n,p} \\ 
\vdots \\ 
\mathbf{I}_{n\times n,2} \\ 
\mathbf{I}_{n\times n,1}%
\end{array}%
\right] =\mathbf{I}_{n\times n\times p}
\end{eqnarray*}%
$.$

Where%
\begin{equation*}
\mathbf{A}_{n\times n\times p}=\left[ 
\begin{array}{c}
\mathbf{A}_{n\times n,p} \\ 
\vdots \\ 
\mathbf{A}_{n\times n,2} \\ 
\mathbf{A}_{n\times n,1}%
\end{array}%
\right] \text{\ and\ }\mathbf{A}_{n\times n\times p}^{-1}=\left[ 
\begin{array}{c}
\mathbf{A}_{n\times n,p}^{-1} \\ 
\vdots \\ 
\mathbf{A}_{n\times n,2}^{-1} \\ 
\mathbf{A}_{n\times n,1}^{-1}%
\end{array}%
\right] .
\end{equation*}
\end{proof}
\end{theorem}

\section{FINDING OF 3D-INVERSE MATRIX}

In this section, we provide a way to find the 3D reverse matrix of a 3D
matrix.

\begin{proposition}
3D inverse matrix of the matrix $\mathbf{A}_{n\times n\times p}\in \mathcal{M%
}_{n\times n\times p}^{\ast }(\mathbf{F}),$ we called the 3D-matrix $\mathbf{%
A}_{n\times n\times p}^{-1}\in \mathcal{M}_{n\times n\times p}^{\ast }(%
\mathbf{F}),$ which has the following form:%
\begin{equation*}
\mathbf{A}_{n\times n\times p}^{-1}=\widehat{\det \left( \mathbf{A}_{n\times
n\times p}\right) }\ast \overline{\overline{\mathbf{A}_{n\times n\times p}}}
\end{equation*}

where $\overline{\overline{\mathbf{A}_{n\times n\times p}}}$ is the \textbf{%
adjugate matrix} of $\mathbf{A}_{n\times n\times p}$ (exactly page by
page)and has the form:%
\begin{equation*}
\overline{\overline{\mathbf{A}_{n\times n\times p}}}=\left[ 
\begin{array}{c}
\left( 
\begin{array}{cccc}
A_{11p} & A_{12p} & \cdots & A_{1np} \\ 
A_{21p} & A_{22p} & \cdots & A_{2np} \\ 
\vdots & \vdots & \ddots & \vdots \\ 
A_{n1p} & A_{n2p} & \cdots & A_{nnp}%
\end{array}%
\right) ^{T} \\ 
\vdots \\ 
\left( 
\begin{array}{cccc}
A_{112} & A_{122} & \cdots & A_{1n2} \\ 
A_{212} & A_{222} & \cdots & A_{2n2} \\ 
\vdots & \vdots & \ddots & \vdots \\ 
A_{n12} & A_{n22} & \cdots & A_{nn2}%
\end{array}%
\right) ^{T} \\ 
\left( 
\begin{array}{cccc}
A_{111} & A_{121} & \cdots & A_{1n1} \\ 
A_{211} & A_{221} & \cdots & A_{2n1} \\ 
\vdots & \vdots & \ddots & \vdots \\ 
A_{n11} & A_{n21} & \cdots & A_{nn1}%
\end{array}%
\right) ^{T}%
\end{array}%
\right]
\end{equation*}

I give a clearer view of matrices $\mathbf{A}_{n\times n\times p}^{-1}$, as
follows:%
\begin{eqnarray*}
\mathbf{A}_{n\times n\times p}^{-1} &=&\widehat{\det \left( \mathbf{A}%
_{n\times n\times p}\right) }\ast \overline{\overline{\mathbf{A}_{n\times
n\times p}}}= \\
&=&\left[ 
\begin{array}{c}
\left( \frac{1}{\left\vert \mathbf{A}_{n\times n,p}\right\vert }\right) \\ 
\vdots \\ 
\left( \frac{1}{\left\vert \mathbf{A}_{n\times n,2}\right\vert }\right) \\ 
\left( \frac{1}{\left\vert \mathbf{A}_{n\times n,1}\right\vert }\right)%
\end{array}%
\right] \ast \left[ 
\begin{array}{c}
\left( 
\begin{array}{cccc}
A_{11p} & A_{12p} & \cdots & A_{1np} \\ 
A_{21p} & A_{22p} & \cdots & A_{2np} \\ 
\vdots & \vdots & \ddots & \vdots \\ 
A_{n1p} & A_{n2p} & \cdots & A_{nnp}%
\end{array}%
\right) ^{T} \\ 
\vdots \\ 
\left( 
\begin{array}{cccc}
A_{112} & A_{122} & \cdots & A_{1n2} \\ 
A_{212} & A_{222} & \cdots & A_{2n2} \\ 
\vdots & \vdots & \ddots & \vdots \\ 
A_{n12} & A_{n22} & \cdots & A_{nn2}%
\end{array}%
\right) ^{T} \\ 
\left( 
\begin{array}{cccc}
A_{111} & A_{121} & \cdots & A_{1n1} \\ 
A_{211} & A_{221} & \cdots & A_{2n1} \\ 
\vdots & \vdots & \ddots & \vdots \\ 
A_{n11} & A_{n21} & \cdots & A_{nn1}%
\end{array}%
\right) ^{T}%
\end{array}%
\right] .
\end{eqnarray*}
\end{proposition}

\begin{proof}
Verified with ease that:%
\begin{equation*}
\mathbf{A}_{n\times n\times p}\odot \mathbf{A}_{n\times n\times p}^{-1}=%
\mathbf{A}_{n\times n\times p}^{-1}\odot \mathbf{A}_{n\times n\times p}=%
\mathbf{I}_{n\times n\times p}.\bigskip
\end{equation*}

Where%
\begin{equation*}
\mathbf{I}_{n\times n\times p}=\left[ 
\begin{array}{c}
\mathbf{I}_{n\times n,p} \\ 
\vdots \\ 
\mathbf{I}_{n\times n,2} \\ 
\mathbf{I}_{n\times n,1}%
\end{array}%
\right] .
\end{equation*}

Let 
\begin{equation*}
\mathbf{A}_{m\times n\times p}\mathbf{=}\left[ 
\begin{array}{c}
\left( 
\begin{array}{cccc}
a_{11p} & a_{12p} & \cdots & a_{1np} \\ 
a_{21p} & a_{22p} & \cdots & a_{2np} \\ 
\vdots & \vdots & \ddots & \vdots \\ 
a_{m1p} & a_{m2p} & \cdots & a_{mnp}%
\end{array}%
\right) \\ 
\vdots \\ 
\left( 
\begin{array}{cccc}
a_{112} & a_{122} & \cdots & a_{1n2} \\ 
a_{212} & a_{222} & \cdots & a_{2n2} \\ 
\vdots & \vdots & \ddots & \vdots \\ 
a_{m12} & a_{m22} & \cdots & a_{mn2}%
\end{array}%
\right) \\ 
\left( 
\begin{array}{cccc}
a_{111} & a_{121} & \cdots & a_{1n1} \\ 
a_{211} & a_{221} & \cdots & a_{2n1} \\ 
\vdots & \vdots & \ddots & \vdots \\ 
a_{m11} & a_{m21} & \cdots & a_{mn1}%
\end{array}%
\right)%
\end{array}%
\right] \in \mathcal{M}_{n\times n\times p}^{\ast }(\mathbf{F}),
\end{equation*}

and 
\begin{eqnarray*}
\mathbf{A}_{n\times n\times p}^{-1} &=&\widehat{\det \left( \mathbf{A}%
_{n\times n\times p}\right) }\ast \overline{\overline{\mathbf{A}_{n\times
n\times p}}}= \\
&=&\left[ 
\begin{array}{c}
\left( \frac{1}{\left\vert \mathbf{A}_{n\times n,p}\right\vert }\right) \\ 
\vdots \\ 
\left( \frac{1}{\left\vert \mathbf{A}_{n\times n,2}\right\vert }\right) \\ 
\left( \frac{1}{\left\vert \mathbf{A}_{n\times n,1}\right\vert }\right)%
\end{array}%
\right] \ast \left[ 
\begin{array}{c}
\left( 
\begin{array}{cccc}
A_{11p} & A_{12p} & \cdots & A_{1np} \\ 
A_{21p} & A_{22p} & \cdots & A_{2np} \\ 
\vdots & \vdots & \ddots & \vdots \\ 
A_{n1p} & A_{n2p} & \cdots & A_{nnp}%
\end{array}%
\right) ^{T} \\ 
\vdots \\ 
\left( 
\begin{array}{cccc}
A_{112} & A_{122} & \cdots & A_{1n2} \\ 
A_{212} & A_{222} & \cdots & A_{2n2} \\ 
\vdots & \vdots & \ddots & \vdots \\ 
A_{n12} & A_{n22} & \cdots & A_{nn2}%
\end{array}%
\right) ^{T} \\ 
\left( 
\begin{array}{cccc}
A_{111} & A_{121} & \cdots & A_{1n1} \\ 
A_{211} & A_{221} & \cdots & A_{2n1} \\ 
\vdots & \vdots & \ddots & \vdots \\ 
A_{n11} & A_{n21} & \cdots & A_{nn1}%
\end{array}%
\right) ^{T}%
\end{array}%
\right] .
\end{eqnarray*}

to prove that it is true:%
\begin{equation*}
\mathbf{A}_{n\times n\times p}\odot \mathbf{A}_{n\times n\times p}^{-1}=%
\mathbf{A}_{n\times n\times p}^{-1}\odot \mathbf{A}_{n\times n\times p}=%
\mathbf{I}_{n\times n\times p}?
\end{equation*}

\bigskip really:%
\begin{equation*}
\mathbf{A}_{n\times n\times p}\odot \mathbf{A}_{n\times n\times p}^{-1}=%
\mathbf{A}_{n\times n\times p}\odot \left( \widehat{\det \left( \mathbf{A}%
_{n\times n\times p}\right) }\ast \overline{\overline{\mathbf{A}_{n\times
n\times p}}}\right) 
\end{equation*}%
\begin{equation*}
=\left[ 
\begin{array}{c}
\left( 
\begin{array}{cccc}
a_{11p} & a_{12p} & \cdots  & a_{1np} \\ 
a_{21p} & a_{22p} & \cdots  & a_{2np} \\ 
\vdots  & \vdots  & \ddots  & \vdots  \\ 
a_{m1p} & a_{m2p} & \cdots  & a_{mnp}%
\end{array}%
\right)  \\ 
\vdots  \\ 
\left( 
\begin{array}{cccc}
a_{112} & a_{122} & \cdots  & a_{1n2} \\ 
a_{212} & a_{222} & \cdots  & a_{2n2} \\ 
\vdots  & \vdots  & \ddots  & \vdots  \\ 
a_{m12} & a_{m22} & \cdots  & a_{mn2}%
\end{array}%
\right)  \\ 
\left( 
\begin{array}{cccc}
a_{111} & a_{121} & \cdots  & a_{1n1} \\ 
a_{211} & a_{221} & \cdots  & a_{2n1} \\ 
\vdots  & \vdots  & \ddots  & \vdots  \\ 
a_{m11} & a_{m21} & \cdots  & a_{mn1}%
\end{array}%
\right) 
\end{array}%
\right] \odot \left[ 
\begin{array}{c}
\left( \frac{1}{\left\vert \mathbf{A}_{n\times n,p}\right\vert }\right)  \\ 
\vdots  \\ 
\left( \frac{1}{\left\vert \mathbf{A}_{n\times n,2}\right\vert }\right)  \\ 
\left( \frac{1}{\left\vert \mathbf{A}_{n\times n,1}\right\vert }\right) 
\end{array}%
\right] \ast \left[ 
\begin{array}{c}
\left( 
\begin{array}{cccc}
A_{11p} & A_{12p} & \cdots  & A_{1np} \\ 
A_{21p} & A_{22p} & \cdots  & A_{2np} \\ 
\vdots  & \vdots  & \ddots  & \vdots  \\ 
A_{n1p} & A_{n2p} & \cdots  & A_{nnp}%
\end{array}%
\right) ^{T} \\ 
\vdots  \\ 
\left( 
\begin{array}{cccc}
A_{112} & A_{122} & \cdots  & A_{1n2} \\ 
A_{212} & A_{222} & \cdots  & A_{2n2} \\ 
\vdots  & \vdots  & \ddots  & \vdots  \\ 
A_{n12} & A_{n22} & \cdots  & A_{nn2}%
\end{array}%
\right) ^{T} \\ 
\left( 
\begin{array}{cccc}
A_{111} & A_{121} & \cdots  & A_{1n1} \\ 
A_{211} & A_{221} & \cdots  & A_{2n1} \\ 
\vdots  & \vdots  & \ddots  & \vdots  \\ 
A_{n11} & A_{n21} & \cdots  & A_{nn1}%
\end{array}%
\right) ^{T}%
\end{array}%
\right] 
\end{equation*}%
\begin{equation*}
\overset{???}{=}\left[ 
\begin{array}{c}
\left( 
\begin{array}{cccc}
a_{11p} & a_{12p} & \cdots  & a_{1np} \\ 
a_{21p} & a_{22p} & \cdots  & a_{2np} \\ 
\vdots  & \vdots  & \ddots  & \vdots  \\ 
a_{m1p} & a_{m2p} & \cdots  & a_{mnp}%
\end{array}%
\right)  \\ 
\vdots  \\ 
\left( 
\begin{array}{cccc}
a_{112} & a_{122} & \cdots  & a_{1n2} \\ 
a_{212} & a_{222} & \cdots  & a_{2n2} \\ 
\vdots  & \vdots  & \ddots  & \vdots  \\ 
a_{m12} & a_{m22} & \cdots  & a_{mn2}%
\end{array}%
\right)  \\ 
\left( 
\begin{array}{cccc}
a_{111} & a_{121} & \cdots  & a_{1n1} \\ 
a_{211} & a_{221} & \cdots  & a_{2n1} \\ 
\vdots  & \vdots  & \ddots  & \vdots  \\ 
a_{m11} & a_{m21} & \cdots  & a_{mn1}%
\end{array}%
\right) 
\end{array}%
\right] \odot \left[ 
\begin{array}{c}
\left( \frac{1}{\left\vert \mathbf{A}_{n\times n,p}\right\vert }\right)
\cdot \left( 
\begin{array}{cccc}
A_{11p} & A_{12p} & \cdots  & A_{1np} \\ 
A_{21p} & A_{22p} & \cdots  & A_{2np} \\ 
\vdots  & \vdots  & \ddots  & \vdots  \\ 
A_{n1p} & A_{n2p} & \cdots  & A_{nnp}%
\end{array}%
\right) ^{T} \\ 
\vdots  \\ 
\left( \frac{1}{\left\vert \mathbf{A}_{n\times n,2}\right\vert }\right)
\cdot \left( 
\begin{array}{cccc}
A_{112} & A_{122} & \cdots  & A_{1n2} \\ 
A_{212} & A_{222} & \cdots  & A_{2n2} \\ 
\vdots  & \vdots  & \ddots  & \vdots  \\ 
A_{n12} & A_{n22} & \cdots  & A_{nn2}%
\end{array}%
\right) ^{T} \\ 
\left( \frac{1}{\left\vert \mathbf{A}_{n\times n,1}\right\vert }\right)
\cdot \left( 
\begin{array}{cccc}
A_{111} & A_{121} & \cdots  & A_{1n1} \\ 
A_{211} & A_{221} & \cdots  & A_{2n1} \\ 
\vdots  & \vdots  & \ddots  & \vdots  \\ 
A_{n11} & A_{n21} & \cdots  & A_{nn1}%
\end{array}%
\right) ^{T}%
\end{array}%
\right] 
\end{equation*}%
\begin{equation*}
\overset{???}{=}\left[ 
\begin{array}{c}
\left( 
\begin{array}{cccc}
a_{11p} & a_{12p} & \cdots  & a_{1np} \\ 
a_{21p} & a_{22p} & \cdots  & a_{2np} \\ 
\vdots  & \vdots  & \ddots  & \vdots  \\ 
a_{m1p} & a_{m2p} & \cdots  & a_{mnp}%
\end{array}%
\right)  \\ 
\vdots  \\ 
\left( 
\begin{array}{cccc}
a_{112} & a_{122} & \cdots  & a_{1n2} \\ 
a_{212} & a_{222} & \cdots  & a_{2n2} \\ 
\vdots  & \vdots  & \ddots  & \vdots  \\ 
a_{m12} & a_{m22} & \cdots  & a_{mn2}%
\end{array}%
\right)  \\ 
\left( 
\begin{array}{cccc}
a_{111} & a_{121} & \cdots  & a_{1n1} \\ 
a_{211} & a_{221} & \cdots  & a_{2n1} \\ 
\vdots  & \vdots  & \ddots  & \vdots  \\ 
a_{m11} & a_{m21} & \cdots  & a_{mn1}%
\end{array}%
\right) 
\end{array}%
\right] \odot \left[ 
\begin{array}{c}
\left( 
\begin{array}{cccc}
a_{11p} & a_{12p} & \cdots  & a_{1np} \\ 
a_{21p} & a_{22p} & \cdots  & a_{2np} \\ 
\vdots  & \vdots  & \ddots  & \vdots  \\ 
a_{m1p} & a_{m2p} & \cdots  & a_{mnp}%
\end{array}%
\right) ^{-1} \\ 
\vdots  \\ 
\left( 
\begin{array}{cccc}
a_{112} & a_{122} & \cdots  & a_{1n2} \\ 
a_{212} & a_{222} & \cdots  & a_{2n2} \\ 
\vdots  & \vdots  & \ddots  & \vdots  \\ 
a_{m12} & a_{m22} & \cdots  & a_{mn2}%
\end{array}%
\right) ^{-1} \\ 
\left( 
\begin{array}{cccc}
a_{111} & a_{121} & \cdots  & a_{1n1} \\ 
a_{211} & a_{221} & \cdots  & a_{2n1} \\ 
\vdots  & \vdots  & \ddots  & \vdots  \\ 
a_{m11} & a_{m21} & \cdots  & a_{mn1}%
\end{array}%
\right) ^{-1}%
\end{array}%
\right] 
\end{equation*}%
\begin{equation*}
=\left[ 
\begin{array}{c}
\left( 
\begin{array}{cccc}
a_{11p} & a_{12p} & \cdots  & a_{1np} \\ 
a_{21p} & a_{22p} & \cdots  & a_{2np} \\ 
\vdots  & \vdots  & \ddots  & \vdots  \\ 
a_{m1p} & a_{m2p} & \cdots  & a_{mnp}%
\end{array}%
\right) \cdot \left( 
\begin{array}{cccc}
a_{11p} & a_{12p} & \cdots  & a_{1np} \\ 
a_{21p} & a_{22p} & \cdots  & a_{2np} \\ 
\vdots  & \vdots  & \ddots  & \vdots  \\ 
a_{m1p} & a_{m2p} & \cdots  & a_{mnp}%
\end{array}%
\right) ^{-1} \\ 
\vdots  \\ 
\left( 
\begin{array}{cccc}
a_{112} & a_{122} & \cdots  & a_{1n2} \\ 
a_{212} & a_{222} & \cdots  & a_{2n2} \\ 
\vdots  & \vdots  & \ddots  & \vdots  \\ 
a_{m12} & a_{m22} & \cdots  & a_{mn2}%
\end{array}%
\right) \cdot \left( 
\begin{array}{cccc}
a_{112} & a_{122} & \cdots  & a_{1n2} \\ 
a_{212} & a_{222} & \cdots  & a_{2n2} \\ 
\vdots  & \vdots  & \ddots  & \vdots  \\ 
a_{m12} & a_{m22} & \cdots  & a_{mn2}%
\end{array}%
\right) ^{-1} \\ 
\left( 
\begin{array}{cccc}
a_{111} & a_{121} & \cdots  & a_{1n1} \\ 
a_{211} & a_{221} & \cdots  & a_{2n1} \\ 
\vdots  & \vdots  & \ddots  & \vdots  \\ 
a_{m11} & a_{m21} & \cdots  & a_{mn1}%
\end{array}%
\right) \cdot \left( 
\begin{array}{cccc}
a_{111} & a_{121} & \cdots  & a_{1n1} \\ 
a_{211} & a_{221} & \cdots  & a_{2n1} \\ 
\vdots  & \vdots  & \ddots  & \vdots  \\ 
a_{m11} & a_{m21} & \cdots  & a_{mn1}%
\end{array}%
\right) ^{-1}%
\end{array}%
\right] 
\end{equation*}%
\begin{equation*}
=\left[ 
\begin{array}{c}
\left( 
\begin{array}{cccc}
1_{F} & 0_{F} & \cdots  & 0_{F} \\ 
0_{F} & 1_{F} & \cdots  & 0_{F} \\ 
\vdots  & \vdots  & \ddots  & \vdots  \\ 
0_{F} & 0_{F} & \cdots  & 1_{F}%
\end{array}%
\right)  \\ 
\vdots  \\ 
\left( 
\begin{array}{cccc}
1_{F} & 0_{F} & \cdots  & 0_{F} \\ 
0_{F} & 1_{F} & \cdots  & 0_{F} \\ 
\vdots  & \vdots  & \ddots  & \vdots  \\ 
0_{F} & 0_{F} & \cdots  & 1_{F}%
\end{array}%
\right)  \\ 
\left( 
\begin{array}{cccc}
1_{F} & 0_{F} & \cdots  & 0_{F} \\ 
0_{F} & 1_{F} & \cdots  & 0_{F} \\ 
\vdots  & \vdots  & \ddots  & \vdots  \\ 
0_{F} & 0_{F} & \cdots  & 1_{F}%
\end{array}%
\right) 
\end{array}%
\right] =\left[ 
\begin{array}{c}
\mathbf{I}_{n\times n,p} \\ 
\vdots  \\ 
\mathbf{I}_{n\times n,2} \\ 
\mathbf{I}_{n\times n,1}%
\end{array}%
\right] =\mathbf{I}_{n\times n\times p}.
\end{equation*}
\end{proof}

\begin{example}
Let's have \ $A_{3\times 3\times 2}=\left[ 
\begin{array}{c}
\left( 
\begin{array}{ccc}
3 & 1 & 5 \\ 
0 & 2 & 1 \\ 
1 & 7 & 4%
\end{array}%
\right)  \\ 
\left( 
\begin{array}{ccc}
1 & 2 & 4 \\ 
8 & 1 & 1 \\ 
3 & 1 & 0%
\end{array}%
\right) 
\end{array}%
\right] \in \mathcal{M}_{3\times 3\times 2}^{\ast }(\mathbb{R}).$ Find its
inverse matrix?

\begin{solution}
The inverse matrix has form:%
\begin{equation*}
\mathbf{A}_{3\times 3\times 2}^{-1}=\widehat{\det \left( \mathbf{A}_{3\times
3\times 2}\right) }\ast \overline{\overline{\mathbf{A}_{3\times 3\times 2}}}
\end{equation*}

we see that%
\begin{equation*}
\det \left( A_{3\times 3\times 2}\right) =\left[ 
\begin{array}{c}
\det \left( 
\begin{array}{ccc}
3 & 1 & 5 \\ 
0 & 2 & 1 \\ 
1 & 7 & 4%
\end{array}%
\right)  \\ 
\det \left( 
\begin{array}{ccc}
1 & 2 & 4 \\ 
8 & 1 & 1 \\ 
3 & 1 & 0%
\end{array}%
\right) 
\end{array}%
\right] =\left[ 
\begin{array}{c}
\left( -6\right)  \\ 
\left( 25\right) 
\end{array}%
\right] \Rightarrow \ \widehat{\det \left( \mathbf{A}_{3\times 3\times
2}\right) }=\left[ 
\begin{array}{c}
\left( -\frac{1}{6}\right)  \\ 
\left( \frac{1}{25}\right) 
\end{array}%
\right] .
\end{equation*}

And 
\begin{equation*}
\overline{\overline{\mathbf{A}_{3\times 3\times 2}}}=\left[ 
\begin{array}{c}
\left( 
\begin{array}{ccc}
1 & 1 & -2 \\ 
31 & 7 & -20 \\ 
-9 & -3 & 6%
\end{array}%
\right) ^{T} \\ 
\left( 
\begin{array}{ccc}
-1 & 3 & 5 \\ 
4 & -12 & 5 \\ 
-2 & 31 & -15%
\end{array}%
\right) ^{T}%
\end{array}%
\right] =\left[ 
\begin{array}{c}
\left( 
\begin{array}{ccc}
1 & 31 & -9 \\ 
1 & 7 & -3 \\ 
-2 & -20 & 6%
\end{array}%
\right)  \\ 
\left( 
\begin{array}{ccc}
-1 & 4 & -2 \\ 
3 & -12 & 31 \\ 
5 & 5 & -15%
\end{array}%
\right) 
\end{array}%
\right] ,
\end{equation*}

Then,%
\begin{eqnarray*}
\mathbf{A}_{3\times 3\times 2}^{-1} &=&\widehat{\det \left( \mathbf{A}%
_{3\times 3\times 2}\right) }\ast \overline{\overline{\mathbf{A}_{3\times
3\times 2}}} \\
&=&\left[ 
\begin{array}{c}
\left( -\frac{1}{6}\right)  \\ 
\left( \frac{1}{25}\right) 
\end{array}%
\right] \ast \left[ 
\begin{array}{c}
\left( 
\begin{array}{ccc}
1 & 31 & -9 \\ 
1 & 7 & -3 \\ 
-2 & -20 & 6%
\end{array}%
\right)  \\ 
\left( 
\begin{array}{ccc}
-1 & 4 & -2 \\ 
3 & -12 & 31 \\ 
5 & 5 & -15%
\end{array}%
\right) 
\end{array}%
\right] =\left[ 
\begin{array}{c}
\left( -\frac{1}{6}\right) \cdot \left( 
\begin{array}{ccc}
1 & 31 & -9 \\ 
1 & 7 & -3 \\ 
-2 & -20 & 6%
\end{array}%
\right)  \\ 
\left( \frac{1}{25}\right) \cdot \left( 
\begin{array}{ccc}
-1 & 4 & -2 \\ 
3 & -12 & 31 \\ 
5 & 5 & -15%
\end{array}%
\right) 
\end{array}%
\right] 
\end{eqnarray*}%
\begin{equation*}
\Rightarrow \mathbf{A}_{3\times 3\times 2}^{-1}=\left[ 
\begin{array}{c}
\left( 
\begin{array}{ccc}
-\frac{1}{6} & -\frac{31}{6} & \frac{3}{2} \\ 
-\frac{1}{6} & -\frac{7}{6} & \frac{1}{2} \\ 
\frac{1}{3} & \frac{10}{3} & -1%
\end{array}%
\right)  \\ 
\left( 
\begin{array}{ccc}
-\frac{1}{25} & \frac{4}{25} & -\frac{2}{25} \\ 
\frac{3}{25} & -\frac{12}{25} & \frac{31}{25} \\ 
\frac{1}{5} & \frac{1}{5} & -\frac{3}{5}%
\end{array}%
\right) 
\end{array}%
\right] 
\end{equation*}

\bigskip 
\end{solution}
\end{example}

\end{document}